\documentclass[12pt]{article}
\usepackage{hyperref}
\usepackage{amsmath, amsthm,enumerate}
\usepackage{amssymb,graphicx}

\usepackage{fullpage,cite}
\usepackage{color}

\usepackage{authblk}
\usepackage{tikz}
\usepackage{color}
\usepackage{float}
\usepackage[utf8]{inputenc}
\usepackage[ruled,vlined,linesnumbered]{algorithm2e}

\tikzset{noddee/.style={circle,draw=black,fill=black,inner sep=2pt}}
\tikzset{nodempty/.style={circle,draw=black,inner sep=2pt}}
\definecolor{red}{RGB}{255,0,0}
\definecolor{blue}{RGB}{0,0,255}
\definecolor{green}{RGB}{0,255,0}

\newtheorem{theorem}{Theorem}[section]
\newtheorem{corollary}[theorem]{Corollary}
\newtheorem{lemma}[theorem]{Lemma}

\newtheorem{proposition}[theorem]{Proposition}

\theoremstyle{definition}
\newtheorem{definition}[theorem]{Definition}

\DeclareMathOperator*{\Exp}{Exp}
\newcommand {\abs}[1] {\vert #1 \vert}
\newcommand {\set}[1]  {\left\{#1\right\}}

\newcommand {\fpt} {{\sf FPT}}
\newcommand {\apx} {{\sf APX}}
\newcommand {\xp} {{\sf XP}}
\newcommand {\p} {{\sf P}}
\newcommand {\np} {{\sf NP}}
\newcommand {\wone} {{\sf W[1]}}
\newcommand {\yes} {{\sf YES}}
\newcommand {\no} {{\sf NO}}
\newcommand {\bigoh}{\mathcal{O}}

\newcommand {\itemref}[1] {\ref{itm:#1}\emph{)}}
\newcommand {\pp}{{\cal P}}
\newcommand {\qq}{{\cal Q}}

\title{On two extensions of equimatchable graphs\thanks{Some results from this work were announced in the Proceedings of the 14th Cologne-Twente
Workshop on Graphs and Combinatorial Optimization~\cite{CTW2016}.}}

\author[1]{Zakir Deniz\thanks{zakirdeniz@sdu.edu.tr}}
\author[2]{T\i naz Ekim\thanks{tinaz.ekim@boun.edu.tr}}
\author[3,4]{Tatiana Romina Hartinger\thanks{tatiana.hartinger@iam.upr.si}}
\author[3,4]{Martin Milani\v c\thanks{martin.milanic@upr.si}}
\author[5]{Mordechai Shalom\thanks{cmshalom@telhai.ac.il}}
\affil[1]{\normalsize Department of Mathematics, S{\"u}leyman Demirel University, Isparta, Turkey}
\affil[2]{\normalsize Department of Industrial Engineering, Bo\u{g}azi\c{c}i University, Istanbul, Turkey}
\affil[3]{\normalsize University of Primorska, UP IAM, Muzejski trg 2, SI-6000 Koper, Slovenia}
\affil[4]{\normalsize University of Primorska, UP FAMNIT, Glagolja\v ska 8, SI-6000 Koper, Slovenia}
\affil[5]{\normalsize TelHai College, Upper Galilee, 12210, Israel}

\begin{document}
\maketitle

\begin{sloppypar}
\begin{abstract}
A graph is said to be equimatchable if all its maximal matchings are of the same size.
In this work we introduce two extensions of the property of equimatchability by defining
two new graph parameters that measure how far a graph is from being equimatchable. The first one, called the matching gap, measures the difference between the sizes of a maximum matching and a minimum maximal matching.
The second extension is obtained by introducing the concept of equimatchable sets; a set of vertices in a graph $G$ is said to be equimatchable if all maximal matchings of $G$ saturating the set are of the same size. Noting that $G$ is equimatchable if and only if the empty set is equimatchable, we study the equimatchability defect of the graph, defined as the minimum size of an equimatchable set in it. We develop several inapproximability and parameterized complexity results and algorithms regarding the computation of these two parameters,
a characterization of graphs of unit matching gap, exact values of the equimatchability defect of cycles, and 
sharp bounds for both parameters.

\end{abstract}
\end{sloppypar}

\section{Introduction}
A \emph{matching} is a set of pairwise disjoint edges in a graph. A matching is \emph{maximal} if it is not contained in any other matching and \emph{maximum} if it is of maximum size.
Given a graph $G$, we denote by $\nu(G)$ and $\beta(G)$ the sizes of a maximum matching and of a minimum maximal matching of $G$, respectively.
Among the many notions related to matchings studied in the literature,
the central one for this paper is the notion of equimatchable graphs.
A graph $G$ is said to be \emph{equimatchable} if all its maximal matchings are of the same size, that is, if $\beta(G)=\nu(G)$.
The question of characterizing equimatchable graphs was posed by Gr\"{u}nbaum in $1974$~\cite{Gruenbaum}; in the same year, equimatchable graphs were
studied by Lewin~\cite{Lewin}. In $1984$, they were shown to be polynomially recognizableby Lesk et al.~\cite{LeskPlummerPulleyblank}; they were also a subject of several other and more recent investigations~\cite{Favaron,KPS2003,KP2009,FHV2010,demange_ekim_equi,EK2016}.

A distance-based extension of equimatchability was recently studied by Kartynnik and Ryzhikov~\cite{KR2016}. In this paper we generalize the property of equimatchability of graphs in two further ways by introducing two graph parameters measuring how far a graph is from being equimatchable. The main part of our paper is thus naturally split into two interrelated parts, each dealing with one of these extensions.
Our approach and results can be summarized as follows:

\medskip
\noindent{\bf I. The matching gap of a graph.}
Our first extension is a graph parameter that we term the \emph{matching gap}, denoted by $\mu(G)$ and defined as the difference $\nu(G)-\beta(G)$. In other words, the matching gap measures the length of the interval of the sizes of maximal matchings of the graph. Clearly, $\mu(G)=0$ if and only if $G$ is equimatchable. We study this parameter in Section~\ref{sec:AlmostEqm}. We characterize graphs of unit matching gap, that is, graphs whose maximal matchings are all of maximum size or one less. We then establish a strong inapproximability result about the problem of computing the matching gap as well as $\wone$-hardness of the problem for a parameter larger than the matching gap. Subsequently, $\xp$ algorithms are provided for this parameter and two other smaller parameters.

\medskip
\vbox{\noindent{\bf II. Equimatchable sets.}
In Section \ref{sec:EqmSets}, we extend the notion of equimatchability from graphs to subsets of vertices of a graph.
We define a set of vertices in a graph to be equimatchable if all maximal matchings of the graph saturating the set are of the same size. We introduce a graph invariant called \emph{equimatchability defect} and denoted by $\eta(G)$, which measures the smallest size of an equimatchable set of a graph. Clearly, $\eta(G)=0$ if and only if $G$ is equimatchable. In Section~\ref{sec:EqmSets-hardness}, we give a hitting set formulation of this parameter and based on it present a reduction from the vertex cover problem showing that computing the equimatchability defect is \apx-hard. On the positive side, we provide an $\xp$ algorithm. In Section~\ref{sec:EqmSets-cycles}, we determine the exact value of $\eta$ for cycles. In Section~\ref{sec:EqmSets-bounds} we give a sharp, polynomial-time computable upper bound for the equimatchability defect in terms of the matching number of the graph.}

Our results on equimatchable sets and the equimatchability defect are based on relations between the equimatchability defect, the vertex cover number, and the clique number, as well the following variant of matching extendability introduced recently by Costa et al.~in~\cite{CdWP}: a graph $G$ is said to be
\emph{expandable} if for every two non-adjacent vertices $u$, $v$ of $G$, the graph $G-u-v$ has a perfect matching.

\medskip
We further connect these two extensions of equimatchability to each other by showing that the matching gap of a graph can be bounded from above by
its equimatchability defect. On the other hand, we observe that the equimatchability defect cannot be bounded from above by any function of the matching gap.

We conclude the paper in Section \ref{sec:Conclusion} with some open questions and further research directions arising from this work.

\bigskip
\noindent{\bf Definitions and Notation.}
The following are some graph theoretical definitions used throughout the paper. Given a graph $H$, we say that a graph is \emph{$H$-free} if it contains no induced subgraph isomorphic to $H$.
For a vertex $v$, the open neighborhood of $v$ in a subgraph $H$ is denoted by $N_H(v)$.
For a subset $U \subseteq V$, $N(U)$ is the union of the neighborhoods of the vertices in $U$.
An edge between two vertices $x$ and $y$ is denoted as $xy$.
We denote paths and cycles as sequences of vertices or as sequences of edges wherever appropriate.
Given two graphs $G_1$ and $G_2$, their {\it disjoint union} is the graph $G_1+G_2$ with vertex set $V(G_1)\cup V(G_2)$ (disjoint union) and edge set $E(G_1)\cup E(G_2)$ (if $G_1$ and $G_2$ are not vertex-disjoint, we first replace one of them with a disjoint isomorphic copy).
We write $kG$ for the disjoint union of $k$ copies of $G$. The {\it join} of two graphs $G_1, G_2$, denoted by $G_1 \ast G_2$, is the graph obtained by adding all the possible edges between vertices of $G_1$ and vertices of $G_2$ to their disjoint union.
For a matching $M$ in a graph $G$ we denote by $V(M)$ the set of vertices of $G$ \emph{saturated} (or \emph{covered}) by $M$. A vertex $v \in V(G)$ not saturated by $M$ is called \emph{exposed}. We denote by $\Exp(M)$ the set of vertices left exposed by $M$.
A matching $M$ is said to \emph{isolate} an independent set $I$ if every vertex of $I$ is a component of $G - V(M)$.
(Equivalently, if no vertex of $I$ is covered by $M$, while every vertex in $N(I)$ is.)
For a graph $G$, we denote by $\nu(G)$ its {\it matching number}, that is, the maximum size of a matching of $G$, and by $\beta(G)$ its {\it MMM number}, that is, the minimum size of a maximal matching of $G$.

Given a matching $M$ in a graph $G$, a path $P=e_1 \ldots e_k$ is $M$-\emph{alternating} if
either ($e_i \in M$ if and only if $i$ is odd) or ($e_i \in M$ if and only if $i$ is even).
An $M$-\emph{alternating cycle} is an even cycle of $G$ where half of its edges are from $M$. It is well known that the symmetric difference of two matchings $M, M'$ consists of $M$-alternating paths and $M$-alternating cycles. An $M$-\emph{decreasing} (resp., $M$-\emph{augmenting}) path is an $M$-alternating path with both endpoints covered (resp., left exposed) by $M$.
A well-known characterization of maximum matchings due to Berge states that a matching $M$ in a graph is maximum if and only if $G$ has no $M$-augmenting paths~\cite{MR0094811}.

The following lemma will be useful in our proofs.

\begin{lemma}\label{lem:set-matchings}
Let $M$ be a maximal matching in a graph $G$. Then for every $k \in [\abs{M},\nu(G)]$ there exists a maximal matching $M'$ in $G$ of size $k$ covering $V(M)$.
\end{lemma}

\begin{proof}
It suffices to show that if $G$ has a maximal matching $M$ of size $k < \nu(G)$ covering a set $S$, then it also has a maximal matching of size $k+1$ covering $S$. Since $M$ is not maximum, $G$ contains an $M$-augmenting path $P$ whose endpoints $u,v$ are exposed by $M$. Let $M' = M \triangle P$ be the matching obtained by applying this augmenting path to $M$. Then
$V(M')=V(M)\cup\{u,v\}$, that is, $M'$ covers $S$. Moreover, since $M$ is maximal, $\Exp(M)$ is an independent set. Therefore, $\Exp(M')=\Exp(M)\setminus\{u,v\}$ is an independent set, implying that $M'$ is maximal.
\end{proof}

\begin{corollary}\label{cor:matchings}
Let $G$ be a graph and $k \in [\beta(G), \nu(G)]$. Then, $G$ has a maximal matching of size $k$.
\end{corollary}

\noindent{\bf Approximation and Parameterized Complexity.}
$\apx$ is the class of problems approximable in polynomial time to within some constant. A problem $\Pi$ is \emph{\apx-hard} if every problem in $\apx$ reduces to $\Pi$ via an approximation-preserving reduction~\cite{MR1734026}. For every \apx-hard problem $\Pi$, there exists a constant $c_\Pi>1$
such that $\Pi$ cannot be approximated to within a factor of $c_\Pi$ unless \p = \np.

Given two optimization problems $\Pi$ and $\Pi '$, we say that $\Pi$ is \emph{$L$-reducible} to $\Pi'$ if there exists a polynomial-time transformation $f$ from instances of $\Pi$ to instances of $\Pi'$ and positive constants $a$ and $b$ such that for every instance $x$ of $\Pi$, we have:
(i) ${\textrm opt}_{\Pi'} (f(x)) \le a\cdot {\textrm opt}_{\Pi}(x)$, and (ii) for every feasible solution $y'$ of $f(x)$ with objective value $c_2$ we can compute in polynomial time a solution $y$ of $x$ with objective value $c_1$ such that $|{\textrm opt}_{\Pi}(x) - c_1| \le b \cdot |{\textrm opt}_{\Pi'}(f(x)) - c_2|$. To show that a problem is \apx-hard, it suffices to show that an \apx-hard problem $\Pi'$
is $L$-reducible to it~\cite{MR1734026}.

An instance of a parameterized problem is a pair $(I,k)$ where $I$ is the input for a decision problem and $k$ is a parameter. An $\fpt$ algorithm for a parameterized problem is one that solves the problem in time $f(k) \cdot \abs{I}^c$ for some constant $c \geq 1$ where $\abs{I}$ is the size of the input $I$ and $f$ is any function. An $\xp$ algorithm is an algorithm that solves the problem in polynomial time for every fixed value of $k$. A  parameterized problem is in $\fpt$ (resp., $\xp$) if it admits an $\fpt$ (resp., $\xp$) algorithm.

An $\fpt$ reduction transforms every instance $(I,k)$ of a parameterized problem $\Pi$ to an instance $(I',k')$ of another parameterized problem $\Pi'$ in $\fpt$ (i.e., $f(k) \cdot \abs{I}^c$) time such that $I'$ is a $\yes$ instance if and only if $I$ is a $\yes$ instance and $k'$ depends only on $k$.
If a problem $\Pi$ can be $\fpt$-reduced to a problem $\Pi'$ and $\Pi'$ admits an $\fpt$ then $\Pi$ admits an $\fpt$. In the infinite W hierarchy of complexity classes, if a (parameterized) problem is $\wone$-hard then it does not admit an $\fpt$ unless $\p = \np$. In order to show that a problem is $\wone$-hard it is sufficient to show that another $\wone$-hard problem $\Pi'$ is $\fpt$-reducible to it.
For further background in parameterized complexity we refer to~\cite{MR3380745,DowneyF13}.

\section{The matching gap of a graph}\label{sec:AlmostEqm}
Our first extension of equimatchability deals with the following natural graph parameter.
Recall that $\nu(G)$ is the matching number of $G$ and $\beta(G)$ is the MMM number of $G$.

\begin{definition}\label{def:matching-gap}
The \emph{matching gap} of a graph $G$ is denoted by $\mu(G)$ and defined as the difference $\nu(G)-\beta(G)$.
\end{definition}

Clearly, for every graph $G$ we have $\mu(G) \ge 0$, with equality if and only if $G$ is equimatchable.
Consequently, the parameter $\mu(G)$ is a measure of how far $G$ is from being equimatchable.

\subsection{A characterization of almost equimatchable graphs}

In this section we characterize graphs with $\mu(G) = 1$, which we term \emph{almost equimatchable}. We will give a short direct proof, though the given characterization can also be derived from the characterization of well-covered graphs due to Tankus and Tarsi~\cite{MR1438624}.
Given a maximal matching $M$ in a graph $G$, we say that $M$ is \emph{with an augmenting $P_4$} if
$G$ has an $M$-augmenting path isomorphic to a $P_4$.

\begin{lemma}\label{lem:pair P4 M}
Let $G$ be a graph and let $k$ be a positive integer with $\beta(G)\le k<\nu(G)$.
Then $G$ has a maximal matching of size $k$ with an augmenting $P_4$.
\end{lemma}

\begin{proof}
Let $M'$ be a maximal matching of $G$ such that $|M'| = k$. Let $S = \Exp(M')$. Since $M'$ is not maximum, we have $\abs{S} \ge 2$.
Consider a shortest $M'$-augmenting path $P'=x_1 x_2 \ldots x_\ell$. Then $\ell$ is even and also $\ell \geq 4$. Assume $\ell \neq 4$.
We observe that $x_1, x_\ell \in S$. Moreover, $N(x_i) \cap S \subseteq \{x_\ell\}$ for each $i \in \set{4,6,\ldots,\ell-2}$, since otherwise there would be an $M'$-augmenting path shorter than $P'$ (namely, $Q=x x_i x_{i+1} \ldots x_\ell$ where $x \in (S \cap N(x_i)) \setminus \set{x_\ell}$).
This implies that in particular $x_1 x_4 \notin E(G)$. Let $M$ be the matching obtained by shifting $M'$ one step towards $x_\ell$ as follows: $M=(M' \setminus \{ x_4 x_5,\ldots,x_{\ell-2} x_{\ell-1} \})\cup \{x_5 x_6,\ldots,x_{\ell-1} x_\ell\}$. Clearly, $\abs{M}=\abs{M'}$. Moreover, $\Exp(M) = (S \setminus\{x_\ell\})\cup\{x_4\}$. We observe that $\Exp(M)$ is an independent set.
Therefore $M$ is a maximal matching of $G$ of size $k$.
Moreover, since $x_2 x_3 \in M$ and $\{x_1,x_4\}\subseteq \Exp(M)$, we infer that $x_1 x_2 x_3 x_4$ is an $M$-augmenting $P_4$.
\end{proof}

\begin{lemma}\label{lem: almost equim > }
For any graph $G$, we have $\mu(G)\geq 1$ if and only if $G$ has a maximal matching with an augmenting $P_4$.
\end{lemma}

\begin{proof}
Assume $\mu(G) \geq 1$. Then, by applying Lemma~\ref{lem:pair P4 M} to a maximal matching of $G$ that is not of maximum size, we conclude that
$G$ has a maximal matching with an augmenting $P_4$.

Now, suppose that $G$ has a maximal matching, say $M$, with an augmenting $P_4$, say $u w y v$ (see Figure~\ref{fig:matching  isolating}).
Then $G$ has two maximal matchings of different sizes, namely $M$ and $M'=(M \setminus \set{wy})  \cup \set{uw,vy}$. Hence, $\mu(G) \geq 1$.
\end{proof}

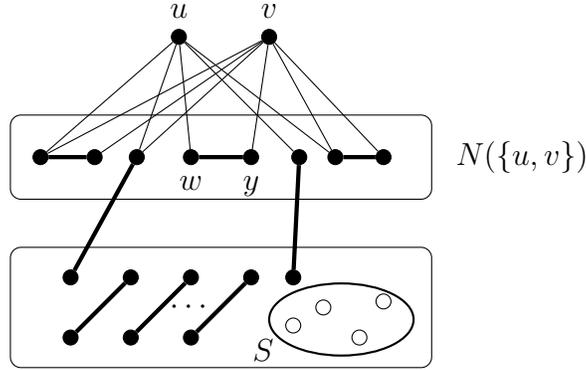
\begin{figure}[h]
\begin{center}
\begin{tikzpicture}[scale=.8]
\draw [rounded corners] (0,.3) rectangle (7,1.7);
\node at (8.5,1) (z) {$N(\{u,v\})$};
\node [noddee] at (.5,1) (x1)  {};	
\node [noddee] at (1.4,1) (x2)  {}
	edge [ultra thick] (x1);	
\node [noddee] at (2.1,1) (x3) {};			
\node [noddee] at (3,1) (x4) [label=below:$w$] {};
\node [noddee] at (4,1) (x5) [label=below:$y$] {}
	edge [ultra thick] (x4);		
\node [noddee] at (4.8,1) (x6) {};
\node [noddee] at (5.4,1) (x7) {};
\node [noddee] at (6.2,1) (x8)  {}
	edge [ultra thick] (x7);
\node [noddee] at (2.8,3) (u) [label=$u$] {}
	edge [] (x1)
	edge [] (x3)
	edge [] (x4)
	edge [] (x6)
	edge [] (x7);
\node [noddee] at (4.3,3) (v) [label=$v$] {}
	edge [] (x1)
	edge [] (x2)
	edge [] (x3)
	edge [] (x5)
	edge [] (x7)
	edge [] (x8);			

\draw [rounded corners] (0,-2.5) rectangle (7,-.5);

\node [noddee] at (1,-1) (p1) {}
	edge [ultra thick] (x3);
\node [noddee] at (1,-2) (y1)  {};
\node [noddee] at (2,-1) (y2)  {}
	edge [ultra thick] (y1);	

\node [noddee] at (2,-2) (y3)  {};
\node [noddee] at (3,-1) (y4)  {}
	edge [ultra thick] (y3);	
\node at (3,-1.5) () {$\cdots$};
\node [noddee] at (3,-2) (y5)  {};
\node [noddee] at (4,-1) (y6)  {}
	edge [ultra thick] (y5);	
\node [noddee] at (4.7,-1) (p2) {}
	edge [ultra thick] (x6);
\draw[thick] (5.5,-1.7) ellipse (1.2cm and .6cm);	

\node [nodempty] at (5.8,-2) (p3) {};
\node [nodempty] at (5.2,-1.5) (p3) {};
\node [nodempty] at (4.7,-1.8) (p3) {};
\node [nodempty] at (6.2,-1.4) (p3) {};
\node at (4.2,-2.2) () {$S$};

\end{tikzpicture}
\end{center}
\caption{A maximal matching $M$, depicted in bold, with an augmenting $P_4$, namely $u w y v$, where $S = \Exp(M)\setminus\{u,v\}$.}
\label{fig:matching  isolating}
\end{figure}

\begin{corollary}\label{cor:equim}
A graph $G$ is equimatchable if and only if it does not contain a maximal matching with an augmenting $P_4$.
\end{corollary}

Now, we characterize graphs of matching gap at most one.

\begin{lemma}\label{lem: almost equim < }
Let $G$ be any graph. Then $\mu(G)\leq 1$ if and only if every two
maximal matchings with an augmenting $P_4$ have the same size.
\end{lemma}

\begin{proof}
Let us first prove necessity of the condition. Suppose that the condition does not hold, i.e.,
$G$ has two maximal matchings $M$ and $M'$ with augmenting $P_4$s respectively $P=u w y v$ and $P'=u' w' y' v'$, and such that $\abs{M} < \abs{M'}$. We observe that $G$ has three maximal matchings of different sizes, namely
$M$, $M'$, and $M'' := (M' \setminus \set{w' y'}) \cup \set{u' w', y'v'}$. (Indeed,
$\abs{M} < \abs{M'} < \abs{M''}$ and $M''$ is maximal.) Therefore, $\mu(G) \geq 2$.

Now, we prove sufficiency. First, note that if there is no
maximal matching with an augmenting $P_4$, then we have $\mu(G)=0$ by Corollary~\ref{cor:equim}.
Otherwise, consider a
maximal matching $M$ with an augmenting $P_4$, say $u w y v$.
Then $G$ has two maximal matchings $M$ and $M'=(M \setminus \set{wy}) \cup \set{uw, vy}$ with $\abs{M} < \abs{M'}$.

It is sufficient to show that $G$ does not contain a maximal matching smaller than $M$ or bigger than $M'$. Suppose that there is a maximal matching $M_1$ in $G$ such that $\abs{M_1} < \abs{M}$. Consider a maximal matching $M_1'$ with an augmenting $P_4$ such that
$\abs{M_1'}=\abs{M_1}$. (The existence of such a matching is guaranteed by Lemma~\ref{lem:pair P4 M}.)
Then $M_1'$ and $M$ are two maximal matchings of $G$ with an augmenting $P_4$
such that $\abs{M_1'} < \abs{M}$, contradicting our assumption.

Now suppose that there is a maximal matching bigger than $M'$. Then $M'$ is not a maximum matching and therefore, by Lemma~\ref{lem:pair P4 M},
we infer that $G$ contains a maximal matching $M''$ with an augmenting $P_4$ such that $\abs{M''}=\abs{M'}$.
Now $M$ and $M''$ are two maximal matchings of $G$ with an augmenting $P_4$
such that $\abs{M''} > \abs{M}$, contradicting our assumption.
\end{proof}

By Lemma \ref{lem: almost equim > } and Lemma \ref{lem: almost equim < }, we have the following characterization of almost equimatchable graphs.

\begin{theorem}\label{thm: gamma 1}
A graph $G$ is almost equimatchable if and only if the following two conditions hold:
\begin{itemize}
\item[$(i)$] $G$ contains a maximal matching with an augmenting $P_4$.
\item[$(ii)$] Every two maximal matchings with an augmenting $P_4$ have the same size.
\end{itemize}
\end{theorem}

At this point it might be worth pointing out that Theorem \ref{thm: gamma 1} cannot be generalized to graphs with matching gap bigger than $1$, using vertex-disjoint augmenting $P_4$s for a maximal matching. Consider the graph $G$ shown in Figure~\ref{fig:no induced P4}. It can be seen that $\mu(G)=2$ since $G$ admits a perfect matching and a minimum maximal matching with four vertices exposed (shown by bold edges in Figure \ref{fig:no induced P4}). However, it can be checked that there is no maximal matching in $G$ with two vertex-disjoint augmenting $P_4$s.

\begin{figure}[htbp]
\begin{center}
\begin{tikzpicture}[scale=.6]
\node [noddee] at (0,0) (v1)  {};
\node [noddee] at (1,1) (v2)  {}
	edge [] (v1);
\node [noddee] at (2,0) (v3)  {}
	edge [ultra thick] (v2);
\node [noddee] at (3,1) (v4)  {}
	edge [] (v3);
\node [noddee] at (4,0) (v5)  {}
	edge [] (v3)
	edge [] (v4);
\node [noddee] at (5,1) (v6)  {}
	edge [] (v5);
\node [noddee] at (6,0) (v7)  {}
	edge [ultra thick] (v6);
\node [noddee] at (7,1) (v8)  {}
	edge [] (v7);
\node [noddee] at (3,2.5) (v9)  {}
	edge [ultra thick] (v4);
\node [noddee] at (4,3.5) (v12)  {}
	edge [] (v9);			
\end{tikzpicture}
\end{center}
\caption{A graph $G$ with $\mu(G)=2$. The edges depicted in bold form the unique minimum maximal matching of $G$.}
\label{fig:no induced P4}
\end{figure}
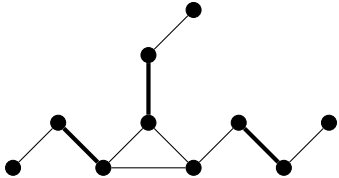

\subsection{The parameterized complexity of computing the matching gap}\label{sec:FPT}

\newcommand{\mg}{\textsc{MatchingGap}}
$\mg$ is the decision problem taking a graph $G$ and an integer $k$ as input, with the objective of determining whether the matching gap of $G$ is at least $k$. Since computing $\beta(G)$, the MMM number of a given graph $G$, is \np-hard~\cite{GavrilYanakkakis80}, and computing the matching number $\nu(G)$ is polynomial, $\mg$ is an \np-complete problem. (The problem is in $\np$ since a yes instance can be certified by a pair of maximal matchings, say $M$ and $M'$, such that $|M'|\ge |M|+k$.) In this section, we analyze the parameterized complexity of $\mg$ with respect to various parameterizations.
In this context, the following lemma regarding different parameterizations will be useful. We remark that the lemma holds in a much broader context, but to keep the presentation simple, we restrict ourselves to the formulation of the lemma in terms of $\xp$ algorithms for $\mg$.

\begin{lemma}\label{lem:XP}
Let $\sigma$ be a non-negative graph invariant such that $\mg$ has an $\xp$ algorithm when parameterized by
$k+\sigma(G)$. Then, $\mg$ has an $\xp$ algorithm when parameterized by $\mu(G)+\sigma(G)$.
\end{lemma}

\begin{proof}
Since $\mg$ is polynomial for equimatchable graphs, we may assume that the input graph $G$ satisfies $\mu(G)\ge 1$. 
Let $A$ be an $\xp$ algorithm for $\mg$ when parameterized by $k+\sigma(G)$ and let $f$ be a non-decreasing function such that
$A$ runs in time $O(|V(G)|^{f(k+\sigma(G))})$. Using $A$, we can obtain an $\xp$ algorithm, say $A'$, for $\mg$ when parameterized by $\mu(G)+\sigma(G)$, as follows. We set $i = 1$ and keep increasing the value of $i$ by one, as long as the answer of
algorithm $A$ on input $(G,i)$ is $\yes$. At the end we will have a positive integer $i'$ such that
the answer of algorithm $A$ on input $(G,i')$ is $\yes$
and the answer of algorithm $A$ on input $(G,i'+1)$ is $\no$.
Clearly, $\mu(G) = i'$ and thus we can determine whether $\mu(G)\ge k$.
This concludes the description of the algorithm $A'$.

The correctness of $A'$ follows from the correctness of $A$.
To estimate the running time of $A'$, disregarding polynomial factors, note that the running time 
is dominated by $\bigoh(|V(G)|^{f(i'+1+\sigma(G))})$, which, since $i' = \mu(G)\ge 1$, function $f$ is non-decreasing, and $\sigma(G)\ge 0$, is in
$\bigoh(|V(G)|^{f(2(\mu(G)+\sigma(G)))})$. It follows that $A'$ is an $\xp$ algorithm for $\mg$
when parameterized by $\mu(G)+\sigma(G)$, as desired.
\end{proof}

First we consider the case when the parameter is $\abs{V(G)}/2-\beta(G)$.
We show that the problem is \wone-hard for this parameter, which implies \wone-hardness with respect to the parameter $\mu(G)$. In other words, the problem does not admit $\fpt$ algorithms with respect to any of these two parameters under standard complexity theory assumptions. We complete this result with a simple $\xp$ algorithm for $\mg$ when parameterized by $\abs{V(G)}/2-\beta(G)$. We recall that an $\xp$ algorithm runs in polynomial time for every fixed value of the parameter.
We further investigate the parameterized complexity under other parameters whose values are smaller than $\abs{V(G)}/2-\beta(G)$ and achieve $\xp$ algorithms for these parameters, too.

\begin{sloppypar}
We start with a lemma that describes a reduction to be used in proving the hardness of $\mg$.
Given a graph $G$, we denote by $\alpha(G)$ its independence number, that is, the maximum size of an independent set in $G$. 
\end{sloppypar}

\begin{lemma}\label{lem:ReductionAlphaToMu}
Let $G$ be a graph and $K(G) = (2G) \ast K_{2\abs{V(G)}} $ the join of the disjoint union of two copies of $G$ with a complete graph of
twice the order of $G$. Then we have
\[
\mu(K(G))= \alpha(G).
\]
\end{lemma}

\begin{proof}
Let $n=|V(G)|$. Clearly, $G'=K(G)$ admits a perfect matching, i.e., $\nu(G')=2n$.
Let $G_1 = 2G$ and let $I$ be a maximum independent set of $G_1$.
Note that $\alpha(G_1) = 2\alpha(G)$, in particular, $|I|$ is even.
We construct a maximal matching $M_I$ of $G'$ as follows.
Every vertex not in $I$ is matched to a vertex of the clique. The $\abs{I}$ remaining vertices of the clique are matched to each other with a perfect matching. Now $\Exp(M_I) = I$, therefore $M_I$ is maximal.
Moreover, $\abs{\Exp(M_I)}=|I| = \alpha(G_1) = 2\alpha(G)$, and
$\abs{M_I}=(4n-2\alpha(G))/2 = 2n-\alpha(G)$.
Therefore,
\[
\beta(G') \leq \abs{M_I} = 2n-\alpha(G) = \nu(G')-\alpha(G)\,,\]
implying $\mu(G') \geq \alpha(G)$.

Consider a maximal matching $M'$ of $G'$. Then $\Exp(M')$ is an independent set of even size and, by construction of $G'$, is a subset of $V(G_1)$.
Consequently, $\abs{\Exp(M')} \leq \alpha(G_1)$. By choosing $M'$ as a minimum maximal matching of $G'$ we get
\[
2\alpha(G) = \alpha(G_1) \geq \abs{\Exp(M')} = 4n - 2 \beta(G') = 2 \nu(G') - 2 \beta(G') = 2 \mu(G'),
\]
implying $\mu(G') \leq \alpha(G)$.
\end{proof}

The problem of computing the independence number of a given graph is \wone-hard when parameterized by solution size (see, e.g., \cite{DowneyF13}). Moreover, for any $\epsilon > 0$, there is no polynomial time algorithm approximating the independence number of a given $n$-vertex graph
to within $n^{1-\epsilon}$, unless \p = \np~\cite{MR2403018}.
Therefore, Lemma~\ref{lem:ReductionAlphaToMu} implies the analogous hardness results for the matching gap. In the inapproximability result below, we interpret $\mg$ as the maximization problem having the set ${\cal M}_G$ of all maximal matchings of a given graph $G$ as the set of feasible solutions and assigning the value $f(M)$, where $f(M) = \nu(G)-|M|$, as the objective function value to a maximal matching $M$.
Clearly, for any graph $G$, we have $\mu(G) = \max\{f(M): M\in {\cal M}_G\}$.
Furthermore, we note that $K(G)$ admits a perfect matching, hence $\mu(K(G))=\abs{V(K(G))}/2 - \beta(K(G))$.

\begin{theorem}
$\mg$ is:
\begin{enumerate}[i)]
\item \wone-hard when parameterized by the matching gap $\mu(G)$, and
\item inapproximable to within $\abs{V(G)}^{1-\epsilon}$ for every $\epsilon > 0$ (unless \p = \np)
\end{enumerate}
even when $G$ admits a perfect matching. Furthermore, the problem is \wone-hard when parameterized by $\abs{V(G)}/2 - \beta(G)$.
\end{theorem}

We now proceed with positive results. Recall that in $\mg$, the input consists of a graph $G$ and a positive integer $k$ and the question
is whether $\mu(G)\ge k$.

\begin{sloppypar}
\begin{theorem}
There is an $\xp$ algorithm for $\mg$ when parameterized by $\abs{V(G)}/2 - \beta(G)$.
\end{theorem}
\end{sloppypar}

\begin{proof}
Let $\sigma(G) = \abs{V(G)}/2 - \nu(G)$. Note that $\abs{V(G)}/2 - \beta(G) = \mu(G)+\sigma(G)$.
Lemma~\ref{lem:XP} implies that
it suffices to show that there is an $\xp$ algorithm for $\mg$ when parameterized by
$k+\sigma(G)$.

Note that $\mu(G) \ge k$ if and only if $\beta(G) \leq \nu(G) - k$ if and only if there is a maximal matching $M$ of $G$ with at most $\nu(G) - k$ edges if and only if there is a maximal matching $M'$ of $G$ with exactly $\nu(G) - k$ edges if and only if there is an independent set $I$ of $G$ with $\abs{V(G)}-2(\nu(G) - k)$ vertices such that $G \setminus I$ has a perfect matching. Therefore, the following is a simple $\xp$ algorithm for the problem when parameterized by $k+\sigma(G)$: try every independent set $I$ of size $2(k+\sigma(G)) = \abs{V(G)} - 2 (\nu(G)-k)$ and check whether the graph $G - I$ has a perfect matching; return $\yes$ if and only if one of the trials succeeds.
\end{proof}

Since the value of the parameter $\abs{V(G)}/2 - \beta(G)$ can be very big, we further investigate the parameterized complexity of the problem under other parameters that in some cases may be significantly smaller than $\abs{V(G)}/2 - \beta(G)$. In the rest of this section we prove two theorems, each of which leads to an $\xp$ algorithm for the matching gap problem. The two algorithms are parameterized by two different parameters, both related to the Gallai-Edmonds decomposition theorem, which we now recall. A graph $G$ is said to be {\it factor-critical} if for every vertex $v\in V(G)$, graph $G-v$ has a perfect matching.

\begin{theorem}[Gallai-Edmonds decomposition~\cite{LP09}]\label{thm:GallaiEdmonds}
Let $G$ be a graph, $D(G)$ the set of vertices of $G$ that are not saturated by at least one maximum matching, $A(G)$ the set of vertices of $V(G) \setminus D(G)$ with at least one neighbor in $D(G)$, and $C(G) = V(G) \setminus (D(G) \cup A(G))$. Then:
\begin{enumerate}[i)]
\item the components of $G[D(G)]$ are factor-critical,
\item $G[C(G)]$ has a perfect matching,
\item every maximum matching of $G$ matches every vertex of $A(G)$ to a vertex of a distinct component of $G[D(G)]$.
\end{enumerate}
\end{theorem}

Theorem \ref{thm:GallaiEdmonds} implies that every maximum matching induces a perfect matching on $C(G)$ and leaves at most one exposed vertex at every component of $G[D(G)]$. We denote by $G_{\it AD}$ the bipartite graph obtained from $G$ by removing the vertices of $C(G)$, removing the
edges with both endpoints in $A(G)$, and contracting every component of $G[D(G)]$ to one vertex. Furthermore, let $\rho(G)$ denote the maximum degree of a vertex of $A(G)$ in $G_{\it AD}$. In other words, $\rho(G)$ is the maximum number of components of $G[D(G)]$ that a vertex of $A(G)$ is adjacent to.
Note that $\rho(G)$ is well-defined if and only if $A(G)$ is non-empty, and if this is the case, then $\rho(G)\ge 1$.
To ease the running time analysis of Algorithm~\ref{alg:recognition k-quasi-eqm-two}, we extend the definition of $\rho(G)$ to the case when $A(G) = \emptyset$, by setting $\rho(G) = 1$ in this case.

The algorithms that we present for the matching gap problem are $\xp$ algorithms
for the parameters $\mu(G)+\abs{A(G)}$ and $\mu(G)+\rho(G)$, respectively.

\begin{theorem} \label{thm:CharacOfMatchingGapKWithTwoMatchings}
Let $G$ be a graph and $k$ a non-negative integer. Then $\mu(G) \geq k$ if and only if
$G$ contains an independent set $I$ and two matchings $M_A$ and $M_A^*$ such that
\begin{enumerate}[i)]
\item \label{itm:sizeOfI2} $\abs{I}=2k$,
\item \label{itm:MAIsMinimal} $M_A$ is an inclusion-wise minimal matching saturating $A(G) \setminus I$ and exposing $I$,
\item \label{itm:MAStarCoversA} $M_A^*$ consists of $\abs{A(G)}$ edges each of which joins a vertex $a \in A(G)$ to a vertex in a distinct component of $G[D(G)]$,
\item \label{itm:ComponentHasPerfectMatching} if $X$ is a connected component of $G[C(G)]$ or a connected component of $G[D(G)]$ with a vertex incident to $M_A^*$, then $G[X \setminus (I \cup V(M_A))]$ has a perfect matching $M_X$, and
\item \label{itm:ComponentHasAlmostPerfectMatching} if $X$ is a connected component of $G[D(G)]$ with no vertex incident to $M_A^*$ then $G[X \setminus (I \cup V(M_A))]$ has a matching $M_X$ exposing exactly one vertex $v_X$ where $N(v_X) \cap I = \emptyset$.
\end{enumerate}
\end{theorem}

\begin{proof}
~\\
($\Rightarrow$) Assume that $\mu(G) \geq k$. Then $G$ contains a maximal matching $M$ with $\nu(G)-k$ edges. One can augment $M$ iteratively $k$ times using $k$ augmenting paths to get a maximum matching $M^*$ such that $V(M) \subseteq V(M^*)$, or equivalently $\Exp(M^*) \subseteq \Exp(M)$. Let $I=V(M^*) \setminus V(M)$, and let $M_A$ (resp. $M_A^*$) be the subset of $M$ (resp. $M^*$) that consists of its edges incident to $A(G)$. We now show that $I, M$, and $M^*$ satisfy the claimed properties.

\itemref{sizeOfI2} Clearly, $\abs{I} = \abs{V(M^*) \setminus V(M)} = \abs{V(M^*)} - \abs{V(M)} = 2 \abs{M^*} - 2 \abs{M} = 2k$.

\itemref{MAIsMinimal} $A(G)$ is saturated by $M^*$, thus $A(G) \setminus I$ is saturated by $M$. Since $M_A$ consists of the edges of $M$ incident to $A(G)$, $M_A$ is a minimal matching having this property. Also, $M_A$ exposes $I$, since $M$ exposes $I$.

\itemref{MAStarCoversA} By the Gallai-Edmonds decomposition theorem, $M^*$ saturates $A(G)$ by matching every vertex of it to a vertex of a distinct component of $G[D(G)]$. Since $M_A^*$ consists of the edges of $M$ incident to $A(G)$ the claimed property holds.

\itemref{ComponentHasPerfectMatching} Let $X$ be a component of $G[C(G)]$. By the Gallai-Edmonds decomposition theorem $M^*$ saturates $X$. Therefore, $M$ saturates $X \setminus I$. The vertices in $X \cap V(M_A)$ are matched to vertices of $A(G)$ and all the rest are matched within component $X$. This induces a perfect matching $M_X$ on $G[X \setminus (I \cup V(M_A))]$. Similarly, if $X$ is a component of $G[D(G)]$ with a vertex incident to $M_A^*$ then $X$ is saturated by $M^*$ and we can repeat the same argument for $X$.

\itemref{ComponentHasAlmostPerfectMatching} Let $X$ be a connected component of $G[D(G)]$ with no vertex incident to $M_A^*$. Then $M^*$ does not contain an edge joining $X$ and $A(G)$, since otherwise such an edge would be in $M_A^*$. Therefore, by the Gallai-Edmonds decomposition theorem, $M^*$ leaves exactly one vertex $v_X \in X$ exposed. This implies that $M$ saturates $X \setminus (I \cup\{v_X\})$ and exposes $X \cap (I\cup\{v_X\})$. We now proceed as above. The vertices in $X \cap V(M_A)$ are matched to vertices of $A(G)$ and all the rest are matched within the component $X$. This induces a perfect matching $M_X$ on $G[X \setminus (I \cup V(M_A)\cup\{v_X\})]$. Furthermore since $M$ is maximal and exposes $I\cup\{v_X\}$, $v_X$ is not adjacent to $I$. 

($\Leftarrow$) Suppose that $G$ contains an independent set $I$ and two matchings $M_A$ and $M_A^*$ with the claimed properties. Let $M = M_A \cup \{M_X: X$ is a connected component of $G[C(G) \cup D(G)]\}$. To conclude the proof, we will show that $M$ is a maximal matching with $\nu(G)-k$ edges. Since $V(M_X) \cap V(M_A) = \emptyset$ by definition of $M_X$, set $M$ is a union of vertex-disjoint matchings, thus a matching. We have $\Exp(M) = I \cup \{v_X~|X$ is a connected component of $G[D(G)]$ with no vertex incident to $M_A^* \}$. Since a) $v_X$ is not adjacent to $I$ for every $X$ as above, b) $I$ is an independent set, and c) vertices $v_X$ belong to distinct connected components of $G[D(G)]$, we conclude that $\Exp(M)$ is an independent set, thus $M$ is maximal.
Let $M^*$ be a maximum matching of $G$ obtained by adding to $M_A^*$ a perfect matching in
each component of $G[C(G)]$, a perfect matching of $G[D(G)\setminus V(M_A^*)]$
for each component of $G[D(G)]$ having a vertex in $V(M_A^*)$,
and a near perfect matching in each component of $G[D(G)]$ having no vertices in $V(M_A^*)$.
(A {\it near perfect} matching in a graph $H$ is a matching exposing a unique vertex.)
A component $X$ of $G[D(G)]$ is not incident to $M_A^*$ if and only if $M^*$ does not contain an edge joining $A(G)$ and $X$, if and only if $X$ contains a vertex exposed by $M^*$. Therefore, the number of these components is $\abs{\Exp(M^*)}$. We also recall that $v_X \notin I$ for every such component $X$. We conclude that $\abs{\Exp(M)}=\abs{I}+\abs{\Exp(M^*)}=2k+\abs{\Exp(M^*)}$, i.e. $\abs{M} = \abs{M^*} - k = \nu(G) - k$.
\end{proof}

Theorem \ref{thm:CharacOfMatchingGapKWithTwoMatchings} suggests Algorithm \ref{alg:recognition k-quasi-eqmWithTwoMatchings}, which decides
if the matching gap of a given graph $G$ exceeds a given integer $k$.

\begin{algorithm}[ht]
\label{alg:recognition k-quasi-eqmWithTwoMatchings}
\caption{Deciding if $\mu(G)\ge k$}

\KwIn{A graph $G = (V,E)$, a non-negative integer $k$.}
\KwOut{$\yes$ if $\mu(G) \geq k$, $\no$ otherwise.}

\BlankLine

{
    \nllabel{line1}
    Compute the Gallai-Edmonds decomposition $\set{A(G), C(G), D(G)}$  of $G$\;
    \For {every independent set $I$ of $G$ with $2k$ vertices} {
        \For {every matching $M_A^*$ of size $\abs{A(G)}$ that joins the vertices of $A(G)$ with distinct components of $D(G)$} {
            \For {every minimal matching $M_A$ that saturates $A(G) \setminus I$ and exposes $I$} {
               $GoodGuess \gets True$\;\nllabel{line5}
                \For {every connected component $X$ of $C(G)$} {
                    \If {$G[X \setminus (I \cup V(M_A))]$ does not have a perfect matching} {
                        $GoodGuess \gets False$\;
                    }
                }
                \For {every connected component $X$ of $D(G)$ with a vertex incident to $M_A^*$} {
                    \If {$G[X \setminus (I \cup V(M_A))]$ does not have a perfect matching} {
                        $GoodGuess \gets False$\;
                    }
                }
                \For {every connected component $X$ of $D(G)$ with no vertex incident to $M_A^*$} {
                    \If {not $\exists v_X \in X \setminus (N[I] \cup V(M_A))$ s.t. \\
                        ~~~~~~~~~$G[X \setminus (I \cup V(M_A)\cup\{v_X\})]$ has a perfect matching} {
                            $GoodGuess \gets False$\;
                    }
                }
                \If {$GoodGuess$} {
                 \Return~\yes \;\nllabel{line17}
                }
            }
        }
    }
    \Return~\no\;
}
\end{algorithm}

Algorithm~\ref{alg:recognition k-quasi-eqmWithTwoMatchings} leads to the following.

\begin{theorem}\label{thm:XP-mu-A(G)}
There is an $\xp$ algorithm for $\mg$ when parameterized by $\mu(G)+\abs{A(G)}$.
\end{theorem}

\begin{proof}
First, note that Algorithm \ref{alg:recognition k-quasi-eqmWithTwoMatchings} is an $\xp$ algorithm for $\mg$ when parameterized by $k+\abs{A(G)}$.
Its correctness follows from Theorem~\ref{thm:CharacOfMatchingGapKWithTwoMatchings}.
To estimate the running time, note that the Gallai-Edmonds decomposition of a graph $G$ can be computed in polynomial time~\cite{LP09}.
Furthermore, observe that one iteration of the third for loop (lines~\ref{line5}--\ref{line17}) can be computed in polynomial time. Therefore, neglecting polynomial factors, the running time of the algorithm is proportional to the number of iterations of the third for loop. The outer loop is executed at most $|V(G)|^{2k}$ times. Each of the two inner loops is executed at most $|E(G)|^{\abs{A(G)}} < |V(G)|^{2 \abs{A(G)}}$ times. Therefore, the third loop is executed at most $|V(G)|^{2k+4\abs{A(G)}} \leq |V(G)|^{4(k+\abs{A(G)})}$ times.
Lemma~\ref{lem:XP} implies that there is an $\xp$ algorithm for $\mg$ when parameterized by $k+\abs{A(G)}$, as claimed.
\end{proof}

We now show that even when the set $A(G)$ is big, we can have an efficient algorithm if every vertex of $A(G)$ is connected only to a small number of connected components in $G[D(G)]$.

\begin{theorem} \label{thm:CharacOfMatchingGapK}
Let $G$ be a graph, $M^*$ a maximum matching of $G$ and $k$ a non-negative integer. Then $\mu(G) \geq k$ if and only if
there exist two disjoint sets $I,Z \subseteq V(M^*)$ such that
\begin{enumerate}[i)]
\item \label{itm:SizeOfI} $\abs{I}=2k$,
\item \label{itm:IZUIndependent} $I \cup Z \cup (\Exp(M^*) \setminus N(I))$ is an independent set,
\item \label{itm:SizeOfZ} $\abs{Z}=\abs{\Exp(M^*) \cap N(I)}$,
\item \label{itm:ZAtmostOneNeighbour} every vertex of $Z$ has at most one neighbour in $\Exp(M^*) \cap N(I)$, and
\item \label{itm:HHasPerfectMatching} the graph $H = G \setminus (I \cup Z \cup (\Exp(M^*) \setminus N(I)))$ has a perfect matching.
\end{enumerate}
\end{theorem}

\begin{proof}

($\Leftarrow$)
Suppose that there are two disjoint sets $I, Z \subseteq V(M)$ having the claimed properties. Let $T=\Exp(M^*) \cap N(I)$ and $U=\Exp(M^*) \setminus N(I)$, and let $M$ be a perfect matching of $H$. Clearly, $\abs{\Exp(M^*)} = \abs{T} \cup \abs{U}$. $M$ is a matching of $G$ such that $\Exp(M)=I \cup Z \cup U$. We note that by Property~\itemref{IZUIndependent} this is an independent set, thus $M$ is maximal in $G$. Since $I,Z$, and $U$ are pairwise disjoint, we have $\abs{\Exp(M)}= \abs{I}+\abs{Z}+\abs{U}$. We also note that by Property \itemref{SizeOfZ} we have $\abs{Z}=\abs{T}$. Therefore, $\abs{\Exp(M)} = \abs{\Exp(M^*)} + 2k$, i.e., $M$ is a maximal matching of size $\nu(G) - k$, implying $\mu(G) \geq k$.

($\Rightarrow$) Suppose $\mu(G) \geq k$. Then $G$ contains at least one maximal matching with $\nu(G)-k$ edges. Let $M$ be a maximal matching with $\nu(G)-k$ edges such that the symmetric difference $\Delta = M^* \triangle M$ is smallest possible. Since $M^*$ is maximum, there are no \hbox{$M^*$-augmenting} paths in $G$. Furthermore, $\Delta$ does not contain alternating cycles, since in this case one can alternate $M$ using such a cycle to obtain a maximal matching $M'$ such that $M^* \triangle M'$ is smaller than $\Delta$. Therefore, $\Delta$ consists of $k$ many \hbox{$M$-augmenting paths} $\pp = \set{P_1, P_2, \ldots, P_k}$ and $\ell\ge 0$ many \hbox{$M^*$-alternating} even length paths $\qq=\set{Q_1, \ldots, Q_\ell}$ where the paths of $\pp \cup \qq$ are pairwise vertex-disjoint. Let $I$ be the set of endpoints of the paths in $\pp$, and let $z_i$ (resp. $t_i$) be the endpoint of $Q_i$ in $V(M^*)$ (resp. $\Exp(M^*)$), for every $i \in [\ell]$. Finally, let $Z=\set{z_i | i \in [\ell]}$, $T=\set{t_i | i \in [\ell]}$ and $U = \Exp(M) \cap \Exp(M^*)$.

\itemref{SizeOfI} Clearly, $I$ and $Z$ are disjoint subsets of $V(M^*)$ and $I$ has $2k$ vertices.

\itemref{IZUIndependent} Since $\Exp(M) = U \cup I \cup Z$, this is an independent set. On the other hand $U = \Exp(M^*) \setminus T = \Exp(M^*) \setminus N(I)$.

\itemref{SizeOfZ}, \itemref{ZAtmostOneNeighbour} Clearly, $\abs{Z}=\abs{T}$. Next, we observe that $z_i$ is not adjacent to a vertex $t_j \in T$ for $j \neq i$. Indeed, in this case $Q_i$ extended with the edge $z_i t_j$ would be an \hbox{$M^*$-augmenting} path. Therefore, every vertex of $Z$ has at most one neighbour in $T$. It is therefore sufficient to show that $T = N(I) \cap \Exp(M^*)$. Let $v \in N(I) \cap \Exp(M^*)$ and suppose $v \notin T$. Noting that $\Exp(M^*)=U\cup T$, the assumption $v \notin T$ implies that $v \in U$; contradiction since $v$ has a neighbor in $I$ but $U\cup I$ is an independent set as it is exposed by the maximal matching $M$. Therefore, $v \in T$. Conversely, let $t_i \in T$ for some $i \in [\ell]$. Since $t_i \in \Exp(M^*)$, it remains to show that $t_i \in N(I)$. Now suppose for a contradiction that $t_i \notin N(I)$. Then, one can alternate $M$ using the path $Q_i$ to obtain a maximal matching $M'$ such that $\abs{M}=\abs{M'}$ and $M' \triangle M^*$ is smaller than $\Delta$, a contradiction. Therefore, $t_i \in N(I)$ as claimed.

\itemref{HHasPerfectMatching} Since $\Exp(M) = U \cup I \cup Z$, $M$ is a perfect matching of $G \setminus (I \cup Z \cup U)$.
\end{proof}

Theorem \ref{thm:CharacOfMatchingGapK} suggests Algorithm \ref{alg:recognition k-quasi-eqm-two}, which decides
if the matching gap of a given graph $G$ exceeds a given integer $k$.

\begin{algorithm}[H]
\label{alg:recognition k-quasi-eqm-two}
\caption{Deciding if  $\mu(G)\ge k$}

\KwIn{A graph $G = (V,E)$, a non-negative integer $k$.}
\KwOut{$\yes$ if $\mu(G) \geq k$, $\no$ otherwise.}

\BlankLine

{
    \nllabel{line1}
    $M^* \gets $ a maximum matching of $G$\;
    \For {every independent set $I \subseteq V(M^*)$ of $G$ with $2k$ vertices} {
        ${\cal Z} \gets \{z \in V(M^*) \setminus N[I]~|$ \\
~~~~~~~~~~~$z$ has at most one neighbour in $\Exp(M^*) \cap N(I)$,\\
~~~~~~~~~~~$z$ has no neighbour in $\Exp(M^*) \setminus N(I)\}$\;
        \For {every independent set $Z$ of ${\cal Z}$ with $\abs{\Exp(M^*) \cap N(I)}$ vertices} {
            \If {$G \setminus (I \cup Z \cup (\Exp(M^*) \setminus N(I)))$ has a perfect matching} {
                \Return~\yes
            }
        }
    }
    \Return~\no
}
\end{algorithm}

Algorithm \ref{alg:recognition k-quasi-eqm-two} leads to the following.

\begin{theorem}\label{thm:XP-mu-A(G)}
There is an $\xp$ algorithm for $\mg$ when parameterized by $\mu(G)+\rho(G)$.
\end{theorem}

\begin{proof}
First, note that Algorithm \ref{alg:recognition k-quasi-eqm-two} is an $\xp$ algorithm for $\mg$ when parameterized by $k+\rho(G)$.
Its correctness follows from Theorem~\ref{thm:CharacOfMatchingGapK}.
To estimate the running time, note that the number of iterations of the inner for loop constitutes the dominant part of the running time
and every other step takes time polynomial in the input. Therefore, neglecting polynomial factors, the running time of the algorithm is proportional to the number of iterations of the inner loop. We observe that $\abs{\Exp(M^*) \cap N(I)} \leq 2k \rho(G)$. Indeed, every vertex of $I$ that has more than one neighbour in $\Exp(M^*)$ has to be an element of $A(G)$ and the number of these neighbours is at most $\rho(G)$. Therefore, $\abs{Z} \leq 2k \rho(G)$ and the number of iterations of the inner loop is at most $|V(G)|^{2k} \cdot |V(G)|^{2k \rho(G)} = |V(G)|^{2k(\rho(G)+1)} \le |V(G)|^{2(k+\rho(G))^2}$.
Lemma~\ref{lem:XP} implies that there is an $\xp$ algorithm for $\mg$ when parameterized by $k+\rho(G)$, as claimed.
\end{proof}

\section{Equimatchable sets}\label{sec:EqmSets}
In this section, we introduce another graph parameter that measures how far a graph is from being equimatchable.
The parameter relies on the notion of equimatchable sets of a graph, which we now define.

We say that a set $S$ of vertices in $G$ is \emph{matching covered} if $S \subseteq V(M)$ for some matching $M$ of $G$. Note that if a set $S$ is matching covered, then there exists a maximal matching of $G$ covering $S$.

\begin{definition}\label{def:eta}
We say that a set $S\subseteq V(G)$ is \emph{equimatchable in} $G$ if all maximal matchings of $G$ covering $S$ are of the same size.
We will refer to the minimum size of an equimatchable set of $G$ as the \emph{equimatchability defect} of $G$ and denote it by $\eta(G)$.
\end{definition}

Any set $S$ that is not matching covered is vacuously equimatchable. As an example of a matching covered set that is not equimatchable, take $G$ to be the $4$-vertex path $v_1 v_2 v_3 v_4$, and let $S$ be any subset of $\set{v_2, v_3}$. Set $S$ is matching covered but not equimatchable as it is covered by two maximal matchings of different sizes.

Since $V(G)$ is always equimatchable in $G$, the equimatchability defect of $G$ is well defined. By Lemma~\ref{lem:set-matchings}, all inclusion-wise maximal matching covered sets of $G$ are of the same size (namely, $2\nu(G)$). This implies that a set $S\subseteq V(G)$ is equimatchable in $G$ if and only if all maximal matchings of $G$ covering $S$ are maximum.

Since $\eta(G) \ge 0$ for all graphs $G$ and a graph $G$ is equimatchable if and only if $\eta(G) = 0$, the equimatchability defect of $G$ measures how far the graph is from being equimatchable. In the following proposition we collect some basic properties of equimatchable sets and of the
equimatchability defect:
\begin{proposition}\label{prop:equimatchable-basic}
Let $G$ be a graph. Then:
\begin{enumerate}[(1)]
\item\label{lbl:EquimSetMonotonicity} If $S \subseteq S' \subseteq V(G)$ and $S$ is equimatchable in $G$, then so is $S'$.
\item\label{lbl:EquimSetThreeTrivial} The following three conditions are equivalent:
\begin{enumerate}
\item $G$ is equimatchable.
\item $\emptyset$ is equimatchable in $G$ (equivalently, $\eta(G)=0$).
\item Every subset of $V(G)$ is equimatchable in $G$.
\end{enumerate}
\end{enumerate}
\end{proposition}

\begin{proof}
\eqref{lbl:EquimSetMonotonicity}
Suppose that there are two vertex sets $S, S'$ of $G$ such that $S \subseteq S'$, $S$ is equimatchable and $S'$ is not equimatchable. Then, $S'$ is covered by two maximal matchings of $G$, say $M_1$ and $M_2$, of different sizes. Since $M_1$ and $M_2$ also cover $S$, we conclude that $S$ is not equimatchable. A contradiction.

\eqref{lbl:EquimSetThreeTrivial}
By definition $G$ is equimatchable if and only if every two maximal matchings in $G$ are of the same size. This is clearly equivalent to requiring that $\emptyset$ is equimatchable in $G$, which establishes the equivalence between \emph{(a)} and \emph{(b)}.
The equivalence between \emph{(b)} and \emph{(c)} follows from \eqref{lbl:EquimSetMonotonicity}.
\end{proof}

It might be worth pointing out two connections between equimatchable sets and the matching gap of the graph. First, in Section~\ref{sec:EqmSets-bounds} we will establish an inequality relating the minimum size of an equimatchable set in a graph with its matching gap (Theorem~\ref{thm:bound2}). Second, the characterization of almost equimatchable graphs given by Theorem~\ref{thm: gamma 1} can be formulated in terms of equimatchable sets using the following equivalent reformulation of Lemma~\ref{lem: almost equim < } (see Figure~\ref{fig:no induced P4}).

\begin{lemma}
Let $G$ be a graph. Then $\mu(G)\leq 1$ if and only if there exists an integer $k$ such that
for every maximal matching with an augmenting $P_4$ in $G$, say $P = uwyv$,
the set $S = N(\{u,v\})\setminus\{w,y\}$ is equimatchable in the graph $G-V(P)$ and all maximal matchings of $G-V(P)$ covering
$S$ have size $k$.
\end{lemma}

\subsection{Hardness of computing the equimatchability defect}\label{sec:EqmSets-hardness}

\newcommand{\ed}{\textsc{EquimatchabilityDefect}}
\begin{sloppypar}
In the $\ed$ problem, the input is a graph $G$ and the objective is to determine $\eta(G)$.
In this section we develop a hitting set formulation of the equimatchability defect of a graph.
This formulation will form the basis for proving that $\ed$ is \apx-hard.
\end{sloppypar}

A {\it hypergraph} ${\cal H}$ is a pair $(V, {\cal E}) $ where $V = V({\cal H})$ is a finite set of {\it vertices} and ${\cal E} = E({\cal H})$ is a set of subsets of $V$, called {\it hyperedges}. Given a positive integer $k$, a hypergraph ${\cal H}$ is said to be {\it $k$-uniform} if $\abs{e} = k$ for all $e\in E({\cal H})$. A subset $S$ of $V({\cal H})$ is a {\it hitting set} of ${\cal H}$ if $e\cap S\neq\emptyset$ for all hyperedges $e\in E({\cal H})$. We will now show that to every graph $G$ one can associate a hypergraph with vertex set $V(G)$ such that a subset $S$ of vertices of $G$ is equimatchable if and only if $S$ is a hitting set of the derived hypergraph.
To define the hypergraph, we need one more definition. We say that a maximal matching in a graph $G$ is \emph{second best} if it is of size $\nu(G)-1$.
For a graph $G$, let $\Exp_2(G)$ be the hypergraph with vertex set $V(G)$ and hyperedge set $\{ \Exp(M) : M \text{ is a second best matching of } G\}$.

Note that the hypergraph $\Exp_2(G)$ is $(\abs{V(G)}-2\nu(G) +2)$-uniform. (If $G$ is equimatchable, then $\Exp_2(G)$ has no hyperedges.)
In particular, $\Exp_2(G)$ is $2$-uniform if and only if $G$ admits a perfect matching.

\begin{proposition}\label{prop:hitting-set}
A subset $S$ of vertices of a graph $G$ is equimatchable if and only if $S$ is a hitting set of the hypergraph $\Exp_2(G)$.
\end{proposition}

\begin{proof}
Let $S \subseteq V(G)$. $S$ is a hitting set of the hypergraph $\Exp_2(G)$ if and only if for every second best matching $M$ of $G$, $S\cap \Exp(M) \neq \emptyset$.
This is equivalent to the property that $S$ is not covered by any second best matching of $G$.
We will now show that this last property is equivalent to $S$ being an equimatchable set in $G$, which proves the result.
$S$ is not equimatchable in $G$ if and only if there exist two maximal matchings $M_1$ and $M_2$ both covering $S$ such that $\abs{M_1}<\abs{M_2}$. Either $M_1$ is a second best matching or it is smaller,  in which case, by Lemma~\ref{lem:set-matchings}, there is a second best matching covering $S$.
\end{proof}

\begin{corollary}\label{cor:eta-hitting-set}
For every graph $G$, we have
$$\eta(G) = \min \{\abs{S}: S\subseteq V(G) \text{ is a hitting set of the hypergraph } {\Exp}_{2}(G) \} \,.$$
\end{corollary}

Recall that fgiven a graph $G$, we denote by $\alpha(G)$ its independence number. By $\omega(G)$ we denote the clique number of $G$, that is, the maximum size of a clique in $G$. We denote by $\tau(G)$ the vertex cover number of $G$, that is, the minimum size of a vertex cover of $G$. It is well known that these parameters are related by the relations $\tau(G)+\alpha(G)=\abs{V(G)}$ and $\omega(G)=\alpha(\overline{G})$, where $\overline{G}$ denotes the complement of $G$. Following Costa et al.~\cite{CdWP}, we call a graph $G$ \emph{expandable} if for every two non-adjacent vertices $u$, $v$ of $G$, the graph $G-u-v$ has a perfect matching.

\newcommand{\vc}{\textsc{VertexCover}}
In the $\vc$ problem, the input is a graph $G$ and the objective is to determine $\tau(G)$. To show \apx-hardness, we will make an $L$-reduction from the $\vc$ problem, which is \apx-hard~\cite{MR1756204}. The reduction will be based on the following lemma, proved using the hitting set formulation of $\eta(G)$.

\begin{lemma}\label{lem:eta-tau}
Let $G$ be an expandable graph with a perfect matching. Then $\eta(G) = \tau(\overline{G})$.
\end{lemma}

\begin{proof}
Since $G$ has a perfect matching, every second best matching of $G$ exposes exactly two (non-adjacent) vertices, which means that every hyperedge of $Exp_2(G)$ consists of a pair of non-adjacent vertices of $G$. Conversely, for every pair $u,v$ of non-adjacent vertices of $G$, the graph $G - u - v$ has a perfect matching, say $M$. Such a matching $M$, is a second best matching of $G$ with $\Exp(M) = \set{u,v}$, which implies that $\set{u,v}$ is a hyperedge of $Exp_2(G)$. Consequently $\Exp_2(G) = \overline{G}$, and the statement of the proposition follows from Corollary~\ref{cor:eta-hitting-set}.
\end{proof}

\begin{corollary}\label{obs:eta-omega}
Let $G$ be an expandable graph with a perfect matching. Then $\eta(G) = \abs{V(G)}-\omega(G)$.
\end{corollary}

\begin{proof}
We have $\eta(G)=\tau(\overline{G})=\abs{V(\overline{G})}-\alpha(\overline{G})=\abs{V(G)}-\omega(G)$.
\end{proof}

\begin{theorem}\label{prop:eta-np-hard}
$\ed$ is \apx-hard.
\end{theorem}

\begin{proof}
Since $\vc$ is \apx-hard for cubic graphs~\cite{MR1756204}, it suffices to show an $L$-reduction from
$\vc$ in cubic graphs to $\ed$.

Let $G$ be a cubic graph that is an instance of $\vc$.
Since $G$ is cubic, $\abs{V(G)}\ge 4$ is even.
We transform $G$ into an instance $G'$ of $\ed$, in two steps:
\begin{enumerate}[(1)]
\item Subdivide each edge twice, and let $G_1$ be the resulting graph. By an observation due to Poljak~\cite{MR0351881}, we have $\tau(G_1)=\tau(G)+\abs{E(G)}$. Note that $G_1$ admits a partition of its vertex set into an independent set (namely $V(G)$) and a set inducing a perfect matching (the rest).

\item Let $G'$ be the complement of $G_1$. Note that $G'$ admits a partition into a clique ($V(G)$) and a
set inducing the complement of a perfect matching (the rest).
\end{enumerate}

Since $\abs{V(G)}$ is even, and $\abs{V(G')}=\abs{V(G)}+2\abs{E(G)}$, graph $G'$ is of even order and it can be easily seen that it has a perfect matching. Using the fact that $\abs{V(G)} \ge 4$, it can also be verified that $G'$ is expandable. Indeed, the subgraph of $G'$ induced by the original vertices (of $G$) forms a clique on $\abs{V(G)}$ vertices, and the subgraph of $G'$ induced by the newly added (subdivision) vertices form a clique on $2 \abs{E(G)}$ vertices minus a perfect matching. It follows that $G'$ has two types of non-edges; a) between two new vertices $u,v$ and b) between one new and one original vertex. In case a), $G-u-v$ has a perfect matching which consists of a perfect matching of $K_{\abs{V(G)}}$ and a perfect matching of the subgraph $K_{2 \abs{E(G)}-2}$  minus a perfect matching. In case b), let $x$ and $y$ be two original vertices which are adjacent in $G$ and $u,v$ be their subdivision vertices in $G'$. Let also the non-edge under consideration be between $x$ and $u$. Now, a perfect matching of $G'-x-u$ can be obtained as follows: match $v$ and $y$ respectively with two different original vertices (this can be done since $\abs{E(G)} \geq 3$ and thus $\abs{V(G)} \geq 4$ since $\abs{V(G)}$ is even), the remaining graph can be partitioned into a clique on $|V(G)|-4$ original vertices and a clique on $2 \abs{E(G)}-2$ new vertices minus a perfect matching, each of which admits a perfect matching.

By Lemma~\ref{lem:eta-tau}, we have that $\eta(G')=\tau(\overline{G'})$. Since $\overline{G'}=G_1$ and $\tau(G_1)=\tau(G)+\abs{(E(G)}$ we have $\eta(G')=\tau(G)+\abs{E(G)}$. We now complete the proof by showing that the above reduction is an $L$-reduction.
Since $G$ is cubic, every vertex in a vertex cover of $G$ covers exactly $3$ edges,
hence $\tau(G)\geq \frac{|E(G)|}{3}$.
This implies that $\eta(G')=\tau(G)+\abs{E(G)}\le 4\tau(G)$; in particular, the first condition
in the definition of $L$-reducibility is satisfied with $a = 4$.

The second condition in the definition of $L$-reducibility
states that for every equimatchable set $S$ of $G'$
we can compute in polynomial time a vertex cover $C$ of $G$
such that
$|C|-\tau(G)\le b \cdot (|S|-\eta(G'))$ for some $b>0$.
We claim that this can be achieved with $b = 1$.
Let $S$ be an equimatchable set of $G'$.
By Proposition~\ref{prop:hitting-set}, $S$ is a hitting set of the hypergraph
$\Exp_2(G')$. The same arguments as those used in the proof of Lemma~\ref{lem:eta-tau} show that
$\Exp_2(G') = \overline{G'} = G_1$, therefore $S$ is a vertex cover of $G_1$.
The set $S$ can be easily modified into a vertex cover $C$ of $G$ such that
$|C|\le |S|-|E(G)|$. (First, transform $S$ into a vertex cover $C_1$ of $G_1$ with $|C_1|\le |S|$ by assuring that for each
edge of $G$ exactly one of the corresponding subdivision vertices belongs to $C_1$. Second, let $C = C_1\cap V(G)$.)
Therefore, $|C|-\tau(G)\le |S|-|E(G)|-\tau(G) \le |S|-\eta(G')$, as claimed.
\end{proof}

We conclude the subsection with an $\xp$ algorithm for the problem of verifying if a given graph $G$ satisfies $\eta(G)\le k$.

\begin{proposition}
There exists an algorithm running in time $\bigoh((nm)^kn^2m)$ that, given a graph $G$ with $n$ vertices and $m$ edges and an integer $k$,
determines if $\eta(G) \le k$.
\end{proposition}

\begin{proof}
Consider the following algorithm.
First, we generate all $\bigoh(n^k)$ sets $S$ of vertices of size $k$.
For each of them we test whether $S$ is equimatchable in polynomial time as follows:
\begin{itemize}
\item First, we check if $S$ is matching covered. If $S$ is not matching covered, then it is vacuously equimatchable.
\item If $S$ is matching covered, then we test whether it is equimatchable as follows. We generate all $\bigoh(m^k)$ inclusion-wise minimal matchings $M$ covering $S$. (Note that each such matching is of size at most $k$.) For each of them, we test whether $G[\Exp(M)]$ is equimatchable.
If these graphs are not all equimatchable, then $S$ is not equimatchable.
If all these graphs are equimatchable, then $S$ is equimatchable if and only if the values
$\abs{M}+\nu(G[\Exp(M)])$
are the same for all matchings $M$ generated above.
\end{itemize}
To estimate the time complexity of the algorithm, note that for each of the $\bigoh(n^k)$ sets $S$, we need
to verify if $S$ is matching covered. This can be done in time $\bigoh(n^2m)$ by reducing the problem
to an instance of the Maximum Weight Matching problem and applying Edmonds' algorithm~\cite{MR0183532} as follows: give weight 2 to edges with both end-vertices in $S$, weight 1 to those having exactly one end-vertex in $S$ and weight 0 otherwise. Now, there is a matching of $G$ covering $S$ if and only if the maximum weighted matching in this instance has weight $\abs{S}$. Moreover, for each of the $\bigoh(m^k)$ inclusion-wise minimal matchings covering $S$, we compute $G[\Exp(M)]$ and test if it is equimatchable, which can be done in time $\bigoh(n^2m)$ using the algorithm by Demange and Ekim~\cite{demange_ekim_equi}.
The values of $\abs{M}+\nu(G[\Exp(M)])$ can be computed and compared along the way, hence the claimed time complexity follows.
\end{proof}

\subsection{Equimatchability defect of cycles}\label{sec:EqmSets-cycles}

Given that the equimatchability defect is hard to compute for general graphs, it is an interesting question to identify families of (non-equimatchable) graphs for which this parameter can be efficiently computed. In the following theorem, we give exact values of the equimatchability defect for all cycles.

\begin{theorem}
For every integer $n\ge 3$, we have
$$\eta(C_n) = \left\{
                \begin{array}{ll}
                  0, & \hbox{if $n\in \{3,4,5,7\}$;} \\
                  n/2, & \hbox{if $n\ge 6$ and $n$ is even;} \\
                  (n-3)/2, & \hbox{if $n\ge 9$ and $n$ is odd.} \\
                \end{array}
              \right.$$
\end{theorem}

\begin{proof}
Each of the cycles $C_3,C_4,C_5,C_7$ is equimatchable, therefore $\eta(C_n) = 0$ for all $n\in \{3,4,5,7\}$. Throughout this proof, index arithmetic is modulo the length of the cycle.

Consider an even cycle of order at least six, $C = v_0 v_1 \ldots v_{2k-1}$ for $k \ge 3$. It is easy to see that the set $S$ of all even-indexed vertices is an equimatchable set because every matching covering $S$ is a perfect matching of $C$. Therefore, $\eta(C) \leq k$. Now, assume for a contradiction that there is an equimatchable set $S'$ with less than $k$ vertices. Then there are two consecutive vertices $v_\ell,v_{\ell+1} \notin S'$. Consider the path $P= C-\{v_{\ell-1}, v_{\ell}, v_{\ell +1}, v_{\ell+2}\}$ on even number of vertices. We claim that $V(P) \nsubseteq S'$.
If $k\ge 4$, then this follows from the fact that $|V(P')| = 2k-4\ge k>|S'|$. If $k = 3$, then $V(P)\subseteq S'$ would imply that $S'$ consists of two consecutive vertices of the cycle, and no such set is equimatchable in the $6$-cycle. Thus, there is at least one vertex $v_p \in V(P) \setminus S'$. Moreover, for exactly one of $v_\ell$ or $v_{\ell+1}$, say w.l.o.g.~for $v_\ell$, the two components of $C -\set{v_p,v_\ell}$ are paths on even numbers of vertices and thus the union of these two paths admits a perfect matching $M'$. Since $v_p$ and $v_\ell$ are non-adjacent, matching $M'$ is maximal in $C$. Hence $M'$ and any maximum matching $M$ of $C$ covering $V(M')$ are two maximal matchings of $C$ of different sizes both covering $S'$, contradicting the equimatchability of $S'$. We conclude that $\eta(C)=k$.

Now, let us consider an odd cycle of order at least nine, $C = v_0 v_1 \ldots v_{2k}$ for $k \ge 4$. We want to show that $\eta(C) = k-1$. We can easily observe that $S =\{v_{2i+1}: 0\le i\le k-2\}$ is an equimatchable set with $k-1$ vertices, so $\eta(C) \leq k-1$. Assume for a contradiction that $C$ has an equimatchable set of size at most $k-2$. Then, by Proposition~\ref{prop:equimatchable-basic}.\eqref{lbl:EquimSetMonotonicity}, $C$
$S'$ has an equimatchable set of size exactly $k-2\ge 2$. We will show that there is an independent set $I\subseteq V(C)\setminus S'$ with $|I| = 3$ and such that the graph $C - I$ is the disjoint union of three non-empty paths of even order. This will imply the existence of a maximal matching $M$ of $C$ saturating $S'$ with $\abs{M}=k-1$, contrarily  to the equimatchability of $S'$ since it can also be trivially covered with a matching of size $k$, and hence $\eta(C)=k-1$.

We first claim that there exists an induced $2K_2$ in $C-S'$. Let $R=V(C)\setminus S'$. Since $|S'|= k-2$, graph $C-S'$ contains at most $k-2$ components, each of which is isomorphic to some path $P_i$, $i\geq 1$. Assume  $C-S'$ has no $2K_2$. This means that
$C-S'$ has at most one nontrivial component, which, if existing, has at most four vertices.
It follows that $C-S'$ has at most $4+(k-3) = k+1$ vertices, a contradiction with $|R|= k+3$. So the claim holds.
Now, pick an induced copy of $2K_2$ in $C-S'$, say $F$, having at least two neighbours in $S'$; note that this is always possible by shifting the vertices of the $2K_2$ towards the vertices of $S'$ if necessary. Remark that $|R|= k+3\geq 7$ since $k\geq 4$, moreover $|N[F]|\leq 8$ where at most $6$ of these vertices are in $R$. It follows that there exists at least one vertex $v_z$ in $C-N[F]$ such that $v_z\in R$. Without loss of generality, let $\{v_l,v_{l+1},v_t,v_{t+1}\}$ induce $F$ and let $v_z$ be such that $v_{l+1}$ and $v_t$ are the closest
vertices of $2K_2$ to $v_z$ in $C$. We now define an independent set $I$ as follows:
\begin{itemize}
\item $v_z\in I$,
\item $v_{l+1} \in I$ if the distance between $v_z$ and $v_{l+1}$ is odd; $v_l \in I$ otherwise,
\item $v_{t} \in I$ if if the distance between $v_z$ and $v_{t}$ is odd; $v_{t+1} \in I$ otherwise.
\end{itemize}
Clearly, $C-I$ consists of three non-empty paths of even order, implying that $C-I$ contains a perfect matching which is a maximal matching of $C$ saturating $S'$ and of size $k-1$, contradiction to the equimatchability of $S'$. Hence $\eta(C)=k-1$.
\end{proof}

\subsection{Sharp upper and lower bounds for the equimatchability defect}\label{sec:EqmSets-bounds}

In this section we derive two bounds for the equimatchability defect, $\eta(G)$: a sharp, polynomial-time computable upper bound in terms of $\nu(G)$, the matching number of the graph, and a sharp lower bound given by $\mu(G)$, the matching gap of the graph.

A trivial upper bound is given by $\eta(G) \leq 2\nu(G)$, since the set of vertices covered by an arbitrary maximum matching is equimatchable.
As we now show, as soon as $G$ has at least one edge, this bound can be further improved by reducing it by $2$.
The graphs with which we will show sharpness of the bound will be particular Cartesian product graphs. Recall that the \emph{Cartesian product} $G \Box H$ of two graphs $G$ and $H$ is the graph with vertex set $V(G) \times V(H)$ in which two distinct vertices $(u,v)$ and $(u',v')$ are adjacent if and only if
\begin{enumerate}[(a)]
\item $u = u'$ and $v$ is adjacent to $v'$ in $H$, or
\item $v = v'$ and $u$ is adjacent to $u'$ in $G$.
\end{enumerate}

\begin{theorem}\label{thm:bound1}
Every graph $G$ with at least one edge satisfies
$$\eta(G) \leq 2\nu(G)-2.$$
Moreover, this bound is sharp; there are graphs of arbitrarily large equimatchability defect
achieving the bound with equality.
\end{theorem}

\begin{proof}
Let $M$ be a maximum matching of $G$ and let $e$ be an edge of $M$. To show the inequality we will show that the set $S:=V(M\setminus \{e\})$ is equimatchable in $G$. Since $S$ has exactly $2\nu(G)-2$ vertices, every maximal matching of $G$ covering $S$ has at least $\nu(G)-1$ edges. Suppose for a contradiction that $M'$ is a maximal matching of $G$ such that $M'$ covers $S$ and $\abs{M'}=\nu(G)-1$. Then $V(M')=S$ which implies that $M'$ is not maximal, since $M' \cup \{e\}$ is a matching properly containing $M'$. Therefore, $S$ is equimatchable, as claimed.

To show sharpness, observe that by Corollary~\ref{obs:eta-omega} any triangle-free expandable graph with a perfect matching achieves the bound with equality. An infinite family of such graphs is given by prisms over odd cycles of length at least $5$, namely, $C_{2k+1} \Box K_2$ for some $k \ge 2$. These graphs are clearly triangle-free and admit a perfect matching. Moreover, $\nu(G)$ can be arbitrarily large as claimed. It remains to show that these graphs are expandable. Let $G=C_{2k+1} \Box  K_2$ for some $k \geq 2$ where $V(C_{2k+1})=\{v_1,v_2,\ldots,v_{2k+1}\}$ and $V(K_2)=\{u_1,u_2\}$. We also denote $G_i=G[\{(v_1,u_i), (v_2,u_i), \ldots, (v_{2k+1},u_i)\}]$ for $i=1,2$. Consider two non-adjacent vertices $(v_i,u_k), (v_j,u_\ell)$ of $G$ for $i,j \in [2k+1]$ and $k,\ell \in \{1,2\}$. First, let $k=\ell$, say $k=\ell=1$ without loss of generality, then for some $p \notin \set{i, j}$, $G_1 \setminus \{(v_i,u_1), (v_j,u_1),(v_{p},u_1)\}$ has a perfect matching $M_1$. Moreover  $G_2 \setminus \{(v_p,u_2)\}$ has a perfect matching $M_2$. Then,
$M_1 \cup M_2 \cup \{(v_p,u_1)(v_p,u_2)\}$ is a perfect matching in $G \setminus \{(v_i,u_1), (v_j,u_1)\}$. Assume now $k \neq \ell$. Then $G_k \setminus \{(v_i,u_k)\}$ has a perfect matching $M_1$ and $G_\ell \setminus \{(v_j,u_\ell)\}$ has a perfect matching $M_2$. We therefore get a perfect matching $M_1 \cup M_2$ of $G \setminus \{(v_i,u_k), (v_j,u_\ell)\}$. Hence $G$ is expandable.
\end{proof}

We remark that, in contrast with the inequality given by Theorem~\ref{thm:bound1}, the matching number of $G$ cannot be bounded from above by any function of $\eta(G)$. There exist equimatchable graphs with arbitrarily large matching number (for example, complete graphs).

We now turn to the lower bound, which relates our two measures of equimatchability to each other.

\begin{theorem}\label{thm:bound2}
For every graph $G$, we have
$$\mu(G) \le \eta(G).$$
Moreover, this bound is sharp; there are graphs with arbitrarily large values of the two parameters achieving the bound with equality.
\end{theorem}

\begin{proof}
Let $S$ be a minimum equimatchable set of $G$. Let $M$ be a minimum maximal matching of $G$ and let $M'$ be a maximum matching of $G$ covering $S$. Then, the symmetric difference of $M'$ and $M$ contains exactly $\mu(G)$ pairwise vertex-disjoint $M$-augmenting paths.

We claim that for each $M$-augmenting path $P$, at least one endpoint of $P$ must belong to $S$. Suppose that this is not the case, and let $P$ be an
$M$-augmenting path such that no endpoint of $P$ is in $S$. Then the symmetric difference of $M'$ and $P$ would form a maximal matching covering $S$ and bigger than $M$, contradicting the fact that $S$ is an equimatchable set.

We conclude that the number of $M$-augmenting paths is at most $\abs{S}$. Therefore $\mu(G) \le \abs{S} = \eta(G)$, as claimed.

To see that the bound is sharp, consider the graph $G_k = k P_4$ (the disjoint union of $k$ $P_4$s). We have $\nu(G_k) = 2k$ and $\beta(G_k) = k$, hence $\mu(G_k) = k$. We also have $\eta(G_k) = k$ since taking one leaf in each $P_4$ results in an equimatchable set $S$. Note that the example could also be made connected by turning $N(S)$ into a clique.
\end{proof}

We remark that, contrary to the inequality $\mu(G)\le \eta(G)$ established by Theorem~\ref{thm:bound2}, the equimatchability defect $\eta(G)$ cannot be bounded from above by any function of $\mu(G)$. For $k \ge 3$, let $G$ be the graph $\overline{k K_2}$. Then $\alpha(G) = 2$ and consequently $\mu(G) \le 1$. On the other hand, it is easy to see that any pair of non-adjacent vertices is missed by some maximal matching, which implies that $\Exp_2(G) \cong \overline{G} = k K_2$, and thus $\eta(G) = \tau(k K_2) = k$.

In particular, Theorem~\ref{thm:bound2} together with the above family of examples show that the two measures of equimatchability are comparable but not equivalent, in the sense that they are not bounded on the same set of graphs.

\section{Concluding remarks}\label{sec:Conclusion}
Let us conclude with some open questions and future research directions related to the two extensions of equimatchability studied in this paper.

For both extensions, some basic complexity questions remain unanswered. 
For the matching gap, the existence of an $\xp$ algorithm for $\mg$ when parameterized by the solution size remains open. Note that Theorem~\ref{thm:XP-mu-A(G)} shows the existence of such an algorithm if every component of the input graph either has a perfect matching or is factor-critical. More specifically, we do not know whether it is polynomial to recognize graphs of matching gap at least two (or, equivalently, almost equimatchable graphs). For the equimatchable sets, we do not know if one can check in polynomial time whether a given set is equimatchable, whether the problem of determining if $\eta(G) \le k$ is in {\sf NP}, and whether it is in co-{\sf NP}. Moreover, the exact (in)approximability and parameterized complexity status of the problem of computing the equimatchability defect has yet to be determined.

Furthermore, one could study several graph modification problems related to graph parameters measuring how far a given graph is from being equimatchable:
the smallest number of edge additions, edge deletions, edge modifications, or vertex deletions needed to turn a given graph into an equimatchable one. In particular, it would be interesting to study these problems in the context of almost equimatchable graphs.

\subsection*{Acknowledgements}

We are grateful to G\'abor Rudolf for inspiring discussions. The work for this paper was done in the framework of a bilateral project between Slovenia and Turkey, financed by the Slovenian Research Agency (BI-TR/$14$--$16$--$005$) and T\"UB\.{I}TAK (grant no:213M620). The work of T.~R.~H.~is supported in part by the Slovenian Research Agency (research program P$1$-$0285$, research project J$1$-$7051$ and a Young Researchers Grant). The work of M.~M.~is supported in part by the Slovenian Research Agency (I$0$-$0035$, research program P$1$-$0285$, research projects N$1$-$0032$, J$1$-$5433$, J$1$-$6720$, J$1$-$6743$, and J$1$-$7051$).


\bibliographystyle{abbrv}
\bibliography{bibliography}
\end{document}